\documentclass{article} 
\usepackage{amsmath}
\usepackage{amsfonts}
\usepackage{graphicx}
\usepackage{color}
\usepackage[utf8]{inputenc}
\usepackage[T1]{fontenc}
\usepackage[french]{babel}

\makeatletter
\def\@sect#1#2#3#4#5#6[#7]#8{%
  \ifnum #2>\c@secnumdepth
    \let\@svsec\@empty
  \else
    \refstepcounter{#1}%
    \protected@edef\@svsec{\@seccntformat{#1}\relax}%
  \fi
  \@tempskipa #5\relax
  \ifdim \@tempskipa>\z@
    \begingroup
      #6{%
        \@hangfrom{\hskip #3\relax\@svsec}%
          \interlinepenalty \@M #8\@@par}%
    \endgroup
    \csname #1mark\endcsname{#7}%
    \addcontentsline{toc}{#1}{%
      \ifnum #2>\c@secnumdepth \else
        \protect\numberline{\csname the#1\endcsname.}%
      \fi
      #7}%
  \else
    \def\@svsechd{%
      #6{\hskip #3\relax
      \@svsec #8}%
      \csname #1mark\endcsname{#7}%
      \addcontentsline{toc}{#1}{%
        \ifnum #2>\c@secnumdepth \else
          \protect\numberline{\csname the#1\endcsname.}%
        \fi
        #7}}%
  \fi
  \@xsect{#5}}
\def\@seccntformat#1{\csname the#1\endcsname.\quad}

\def\@begintheorem#1#2{\trivlist
   \item[\hskip \labelsep{\bfseries #1\ #2.}]\itshape}
\def\@opargbegintheorem#1#2#3{\trivlist
      \item[\hskip \labelsep{\bfseries #1\ #2\ (#3).}]\itshape}

\makeatother

\newtheorem{theo}[equation]{Th\'eor\`eme}
\newtheorem{lem}[equation]{Lemme}

\newtheorem{definition}[equation]{D\'efinition}
\newtheorem{cor}[equation]{Corollaire}
\newenvironment{remarque}{
\refstepcounter{equation}\trivlist%
\item[\hskip \labelsep{\bfseries Remarque \theequation.\ }]}%
{\endtrivlist}%

\newtheorem{conj}[equation]{Conjecture}
\makeatletter
\renewcommand\theequation{\thesection.\arabic{equation}}
\@addtoreset{equation}{section}
\makeatother
\newcommand{\carrenoir}{\rule{0.5em}{0.5em}}
\makeatletter
\newenvironment{proof}[1][\@empty]{\textbf{D\'emonstration~%
\ifx\@empty#1:\else #1~:\fi~}}
{\hfill\carrenoir\nolinebreak\vspace{2mm}}
\makeatother
\newcommand{\oper}[2]{\newcommand{#1}{\mathop{\mathrm{#2}}\nolimits} }
\oper{\Vol}{Vol}
\newcommand{\de}{\mathrm{ d }}
\newcommand{\R}{\mathbb R}

\title{Prescription du spectre de Steklov dans une classe conforme}
\author{Pierre Jammes}
\date{}
\begin{document}
\maketitle
{\small 
\textsc{Résumé.---}
Sur toute variété compacte de dimension $n\geq3$ à bord, on prescrit 
toute partie finie du spectre de Steklov dans une classe conforme donnée.
En particulier, on prescrit la multiplicité des valeurs propres. Sur une
surface compacte à bord donnée, on montre que la multiplicité de la 
$k$-ième valeur propre est bornée indépendamment de la métrique. Sur le
disque, on donne des résultats plus précis : la multiplicité de la 1\iere{}
et la 2\ieme{} valeurs propres non nulles sont au plus~2 et~3 respectivement.
Pour le problème de Steklov-Neumann sur le disque, on montre que la 
multiplicité de la $k$-ième valeur propre non nulle est au plus $k+1$.

Mots-clefs : spectre de Steklov, prescription de spectre, géométrie conforme.

\medskip
\textsc{Abstract.---}
On any compact manifold of dimension $n\geq3$ with boundary,
we prescribe any finite part of the Steklov spectrum whithin a given 
conformal class. In particular, we prescribe the multiplicity of the
first eigenvalues. On a compact surface with boundary, we show that
the multiplicity of the $k$-th eigenvalue is bounded independently of
the metric. On the disk, we give more precise results : the multiplicity
of the first and second positive eigenvalues are at most~2 and ~3
respectively. For the Steklov-Neumann problem on the disk, we prove
that the multiplicity of the $k$-th positive eigenvalue is at most $k+1$.


}

\section{Introduction}
Étant donnée une variété riemannienne $(M,g)$ compacte à bord et une
fonction strictement positive $\rho\in C^0(\partial M)$,
le spectre de Steklov de $M$ est l'ensemble des réels $\sigma$ tels que
le système
\begin{equation}
\left\{\begin{array}{ll}
\Delta f=0 & \textrm{dans }M\\
\frac{\partial f}{\partial \nu}=\sigma \rho f & \textrm{sur }\partial M
\end{array}\right.
\end{equation}
où $\nu$ est un vecteur normal unitaire sortant le long de $\partial M$, 
admette des solutions non triviales. Ce spectre est formé de valeurs 
propres positives notées $0=\sigma_0(M,g,\rho)<\sigma_1(M,g,\rho)\leq
\sigma_2(M,g,\rho)\ldots\to+\infty$. Si $\rho\equiv 1$, alors c'est aussi
le spectre de l'opérateur Dirichlet-to-Neumann sur $M$.

Un grand nombre de travaux récents visent à borner ces valeurs propres
sous des contraintes géométriques, par exemple avec des hypothèses
conformes (\cite{fs11}, \cite{ha11}), ou en fonction d'une constante
isopérimétrique (\cite{cesg11}). Le but de cet article vise au contraire
à mettre en évidence le fait que ce spectre possède une certaine souplesse
et que si la dimension de $M$ est au moins~3, on peut en prescrire toute 
partie finie, c'est-à-dire que si on se donne une suite finie de réels 
strictement 
positifs, il existe une métrique sur $M$
telle que cette suite soit le début du spectre. On va en fait montrer 
un résultat plus fort, à savoir qu'on peut prescrire simultanément 
le début du spectre, la multiplicité des valeurs propres, la classe conforme
de la variété et la fonction densité $\rho$ sur le bord. On étend ainsi 
au spectre de Steklov un résultat obtenu par Y.~Colin de Verdière pour
le laplacien dans~\cite{cdv87} et généralisé aux formes différentielles
dans~\cite{ja11}.

\begin{theo}\label{intro:th1}
Soit $(M^n,g)$ une variété riemannienne compacte à bord de dimension $n\geq3$.
Étant donnée une fonction strictement positive 
$\rho\in C^0(\partial M)$, un entier $N\geq1$ et une suite
finie de réels strictement positifs $0<a_1\leq a_2\leq\ldots\leq a_N$,
il existe une métrique $\tilde g$ conforme à $g$ telle que
$$\sigma_k(M,\rho,\tilde g)=a_k$$
pour tout $k\in[1,N]$.
\end{theo}
\begin{remarque}
On ne peut pas prescrire simultanément le spectre, le volume et la classe
conforme. En effet, si on fixe le volume et la classe conforme, les
valeurs propres ne peuvent pas être arbitrairement grandes (cf.~\cite{fs11}
et~\cite{ha11}). Des obstructions semblables existent pour le laplacien 
usuel (\cite{esi86}), le laplacien de Hodge en restriction aux formes
différentielles de certains degrés (\cite{ja07}, \cite{ja08}) 
et l'opérateur de Dirac (\cite{am03}). 
Le théorème~\ref{intro:th1} montre en revanche que
même en fixant la classe conforme, on ne peut pas majorer le rapport
$\sigma_k/\sigma_l$ pour $k>l$.
\end{remarque}

\begin{remarque}
On sait que la prescription de multiplicité est possible pour les 
opérateurs de Schrödinger en dimension $n\geq3$ 
(\cite{cdv86}, \cite{cdv87}) et les opérateurs 
agissant sur les formes différentielles en dimension $n\geq4$ (\cite{ja11},
\cite{ja12}). Mais ce problème
n'est toujours pas résolu pour les formes différentielles en dimension~3,
ni pour l'opérateur de Dirac, dont on ne sait actuellement prescrire le 
spectre que si les valeurs propres sont simples (\cite{da05}).
\end{remarque}

 La principale difficulté consiste à prescrire la multiplicité des
valeurs propres. On utilisera pour cela les techniques introduites
par Y.~Colin de Verdière dans \cite{cdv86} (voir~\cite{ja09b} pour une 
présentation plus détaillée de ce sujet). Les principaux ingrédients
sont des théorèmes de convergence spectrale (en particulier de 
convergence du spectre vers celui d'un domaine
de la variété, théorème~\ref{conv:thdom} de la section~\ref{conv}), 
un modèle de valeur propre multiple fourni par un laplacien combinatoire
sur des graphes (paragraphe~\ref{presc:graphe}).

La démonstration du théorème~\ref{intro:th1} échoue en dimension~2, entre
autres à cause de l'invariance conforme de la norme $L^2$ du gradient
d'une fonction. On va montrer que cette difficulté ne peut pas être 
contournée et qu'il existe en fait une obstruction à la prescription
de multiplicité en dimension~2. La démonstration suit celles de 
S.~Y. Cheng \cite{ch76} et G.~Besson \cite{be80} pour majorer la 
multiplicité des valeurs propres du laplacien. 

\begin{theo}\label{intro:th2}
Sur toute surface riemannienne compacte orientable à bord $(M,g)$
de genre $\gamma$ et toute  fonction strictement positive 
$\rho\in C^0(\partial M)$, la multiplicité de $\sigma_k(M,\rho,g)$
est majorée par $4\gamma+2k+1$. Si $M$ est non orientable et qu'on
note $l$ le nombre de composantes connexes de $\partial M$, alors 
la multiplicité de $\sigma_k(M,\rho,g)$ est majorée par $4p+4k+1$,
où $p$ est l'invariant topologique $1-\chi(M)-l$.
\end{theo}
\begin{remarque}
Lors de la finalisation de cet article, la démonstration de bornes sur la 
multiplicité est apparu simultanément dans deux prépublications.
Dans \cite{fs12}, A.~Fraser et R.~Schoen ont montré indépendamment le
même théorème, avec une démonstration presque identique. Ils montrent 
aussi que la
borne obtenue pour $\sigma_1(S^1\times[0,1])$, à savoir~3, est optimale.
Simultanément, M.~Karpukhin, G.~Kokarev, I.~Polterovich ont démontré
dans \cite{kkp12} une amémioration de ces bornes avec des techniques 
différentes : ils montrent que la multiplicité de $\sigma_k$ est majorée
par $2p+2k+1$ et $2p+2l+k$, que la surface soit orientable ou non.
\end{remarque}

Les bornes données par le théorème~\ref{intro:th2} sont les mêmes que celles
obtenues par G.~Besson pour
les valeurs propres de laplacien. Il s'avère que le spectre de Steklov 
possède des rigidités supplémentaires qu'on va illustrer dans le cas 
du disque :

\begin{theo}\label{intro:th3}
Sur le disque $\mathbb D$, la multiplicité de $\sigma_1(\mathbb D,\rho,g)$ est
au plus~2 et celle de $\sigma_2(\mathbb D,\rho,g)$ est au plus~3.
\end{theo}
\begin{remarque}
En utilisant les résultats de~\cite{cdv88}, on peut facilement construire
(par excision d'un petit disque sur la sphère)
une métrique sur $\mathbb D$ telle que la première valeur propre non nulle
du laplacien avec condition de Neumann (ou la seconde pour la condition 
de Dirichlet) soit de multiplicité~3. En outre, la borne sur la 
multiplicité de $\sigma_1$ est optimale puisque pour la métrique canonique, 
toutes les valeurs propres non nulles sont doubles.
\end{remarque}
\begin{remarque}
L'article \cite{am94}, qui traite de la multiplicité des $\sigma_k$ sur 
le disque, contient comme cas particulier le fait que
la multiplicité de $\sigma_1$ est au plus~2; on en donnera une démonstration
un peu plus directe. En revanche, la borne sur la multiplicité de 
$\sigma_2$ ne semble pas être apparue auparavant dans la littérature.
\end{remarque}

 On va aussi montrer une autre borne sur la multiplicité dans le cas du
disque, mais pour une variante du problème de Steklov, à savoir le
problème de Steklov-Neumann. Ce problème est défini de la manière suivante :
on partitionne le bord $\partial M$ en deux sous-variétés (pas nécessairement
connexes) $\partial M=
\partial M_S\cup\partial M_N$ et pour une fonction 
$\rho\in C^0(\partial M_S)$, on pose la condition 
$\frac{\partial f}{\partial \nu}=\sigma \rho f$ sur $\partial M_S$
et on demande à $f$ de vérifier la condition de Neumann sur $\partial M_N$
(voir paragraphe~\ref{rappels:sn} pour plus de détails).
\begin{theo}\label{intro:th4}
Étant donnée une partition (non triviale) $\partial\mathbb D_S\cup
\partial\mathbb D_N$ du
bord du disque $\mathbb D$, la multiplicité de $\sigma_k(\mathbb D,\rho,g)$
pour le problème de Steklov-Neumann relativement à cette partition est au plus 
$k+1$.
\end{theo}
\begin{remarque}
Pour le laplacien, les meilleures bornes connues sont asymptotiquement 
de l'ordre de $2k$ quand $k\to+\infty$. Par exemple, pour le laplacien
de Dirichlet sur le disque, il et montré dans \cite{homn99} que la multiplicité
de la $k$-ième valeur propre est au plus $2k-3$. Dans \cite{kkp12}, la
meilleure borne donnée pour le problème de Steklov sur le disque est $k+2$.
\end{remarque}
\begin{remarque}
Dans le cas particulier du problème hydrodynamique de ballottement 
(voir \cite{kk01} ou les rappels du paragraphe~\ref{rappels:sn}), on 
sait que la première valeur
propre non nulle est simple (cf.~\cite{kkm04}). Il est conjecturé que les
autres sont simples aussi, mais cette question reste ouverte. 
\end{remarque}

Y. Colin de Verdière a conjecturé que la multiplicité maximale de la 
2\ieme{} valeur propre d'un opérateur de Shrödinger sur une surface $M$ est
$\mathrm{Chr}(M)-1$ où $\mathrm{Chr}(M)$ est le nombre chromatique de $M$,
c'est-à-dire le nombre de sommets du plus grand graphe complet plongeable
dans $M$. Comme la démonstration du théorème~\ref{intro:th1} repose sur
des graphes plongées dans $M$ dont les sommets sont sur le bord de la variété, 
on peut envisager de transposer cette conjecture au problème de Steklov 
sous la forme suivante :
\begin{conj}
Sur une surface compacte à bord $M$, la multiplicité maximale de $\sigma_1(M)$ 
est $\mathrm{Chr}(M,\partial M)-1$, où $\mathrm{Chr}(M,\partial M)$ désigne
le nombre de sommets du plus grand graphe complet qu'on peut plonger dans 
$M$ en plaçant les sommets sur $\partial M$.
\end{conj}
D'après ce qui précède, cette conjecture est vérifiée sur le disque 
$\mathbb D$ et le cylindre $S^1\times [0,1]$. 

La section~\ref{rappels} rappellera quelques propriétés du spectre de Steklov
et de l'opérateur Dirichlet-to-Neumann dont nous auront besoin. Nous
montrerons dans la section~\ref{conv} les théorèmes de convergence 
spectrale que nous utiliserons, et nous les appliquerons dans la 
section~\ref{presc} pour démontrer le théorème~\ref{intro:th1}. Enfin,
la section~\ref{bornes} sera consacrée au cas de la dimension~2 et à 
la démonstration des théorèmes~\ref{intro:th2}, \ref{intro:th3} 
et~\ref{intro:th4}.

\section{Le problème de Steklov}\label{rappels}
\subsection{Définition du spectre de Steklov}
On se donne une variété riemannienne $(M,g)$ compacte à bord telle que
$\partial M$ soit $C^1$ par morceau (dans la suite,
$g$ désignera indifféremment la métrique sur $M$ ou la métrique induite sur
$\partial M$). 
Le problème des valeurs propres de Steklov consiste à résoudre l'équation
\begin{equation}
\left\{\begin{array}{ll}
\Delta f=0 & \textrm{dans }M\\
\frac{\partial f}{\partial \nu}=\sigma \rho f & \textrm{sur }\partial M
\end{array}\right.
\end{equation}
où $\nu$ est un vecteur unitaire sortant normal au bord et 
$\rho\in C^0(\partial M)$ un fonction densité fixée. 
%On note
%$\mathcal M(\partial M)=\int_{\partial M}\rho\de v_g$ (si $\rho=1$ on parle
%de problème de Steklov homogène, et on a dans ce cas $\mathcal M(\partial M)
%=\Vol(\partial M,g)$). 
L'ensemble des
réels $\sigma$ solutions du problème forme un spectre discret positif
noté
\begin{equation}
0=\sigma_0(M,g,\rho)<\sigma_1(M,g,\rho)\leq\sigma_2(M,g,\rho)\ldots.
\end{equation}
Le problème de Steklov,
déjà étudié à la fin du XIX\ieme{} siècle et au début du XX\ieme{} 
(voir~\cite{st99}, \cite{st02} et les références qui y sont données), 
apparaît dans divers problèmes physiques. Par exemple il
permet de modéliser l'évolution d'une membrane libre dont la masse se
concentre sur son bord, et il intervient dans certains problèmes de 
tomographie. On verra au paragraphe qui suit qu'il apparaît aussi
en hydrodynamique.

Notre principal outil sera la caractérisation variationnelle suivante
du spectre de Steklov (cf. \cite{ba80}) :
\begin{equation}\label{rappels:minmax}
\sigma_k(M,g,\rho)=\inf_{V_{k+1}\in H^1(M)}
\sup_{f\in V_{k+1}\backslash\{0\}}
\frac{\int_M|\de f|^2\de v_g}%
{\int_{\partial M}f^2\rho\ \de v_g},
\end{equation}
où $V_k$ parcours les sous-espaces de dimension~$k$ de l'espace 
de Sobolev $H^1(M)$. 

 Il faut prendre garde au fait que $|f|^2=\int_{\partial M}f^2\rho\ \de v_g$ 
ne définit pas une norme de Hilbert sur 
$L^2(M)$ (elle est nulle sur les fonctions vérifiant la condition 
de Dirichlet). En revanche, on peut utiliser les techniques usuelles de
min-max en considérant l'espace de Hilbert $L^2(\partial M)$ muni de la
métrique $|\cdot|$ qu'on vient de définir, et la forme quadratique 
$Q(f)=\int_M|\de \tilde f|^2\de v_g$ où $\tilde f$ est le prolongement
harmonique de $f$. Il sera parfois commode de redéfinir la forme quadratique
$Q$ par
\begin{equation}\label{rappels:Q}
Q(f)=\inf_{\stackrel{\tilde f\in H^1(M)}{\tilde f_{|\partial M}=f}}
\int_M|\de \tilde f|^2\de v_g.
\end{equation}
Cette définition sera en particulier applicable dans les situations où
on considère une métrique singulière sur $M$ (voir 
paragraphe~\ref{rappels:sing}).

Dans le cas homogène, c'est-à-dire quand $\rho\equiv1$, le spectre de 
Steklov est aussi connu comme étant
le spectre de l'opérateur Dirichlet-to-Neumann, qu'on notera 
$\Lambda: C^\infty(\partial M)\to C^\infty(\partial M)$, défini comme suit :
étant donné une fonction $f\in C^\infty(\partial M)$, on prolonge 
harmoniquement $f$ dans $M$ et on pose
\begin{equation}
\Lambda f(x)=\frac{\partial f}{\partial \nu}(x).
\end{equation}
Le spectre de $\Lambda$ est bien celui de $Q$ car pour une fonction harmonique,
on a $\int_M|\de f|^2\de v_g=\int_{\partial M}f\frac{\partial f}{\partial \nu}\de v_g$.

L'opérateur $\Lambda$ n'est pas un opérateur différentiel sur $\partial M$
(ce n'est même pas un opérateur local),
mais c'est un opérateur pseudo-différentiel elliptique d'ordre~1 
(cf.~\cite{ta96b}, ch.~7). En particulier, nous utiliserons le fait qu'il
vérifie une inégalité elliptique :
\begin{equation}\label{rappels:ell}
\|f\|_{H^1(\partial M)}^2\leq c 
\int_{\partial M}f\Lambda f\de v_g
+c'\|f\|_{L^p(\partial M)}^2,
\end{equation}
où $p\in[1,+\infty]$, les constante $c$, $c'$ dépendant de $p$ et de 
la métrique $g$ sur $M$ mais pas de $f$.

Pour finir, nous auront besoin d'une propriété d'unique prolongement des
fonctions propres en dimension~2:
\begin{theo}\label{rappels:prol}
Soit $f$ une fonction propre du problème de Steklov sur une surface.
Si $f$ s'annule sur un ouvert du bord, alors $f\equiv0$.
\end{theo}
\begin{proof}
Soit $I$ un intervalle du bord sur lequel $f$ s'annule. On peut déformer
conformément la surface de manière à ce que $I$ devienne géodésique
et que la métrique reste inchangée sur le reste du bord. Par invariance
conforme de l'harmonicité et de la condition $\partial f/\partial \nu=0$,
$f$ est toujours fonction propre. En notant $x$ un paramètre sur $I$,
on a $\partial^2 f/\partial x^2=0$, donc aussi $\partial^2 f/\partial \nu^2=0$
puisque $f$ est harmonique. Enfin, comme $\partial f/\partial \nu=0$ sur 
$I$, on a aussi $\partial^2 f/\partial x\partial \nu=0$ et donc le
développement à l'ordre~2 de $f$ est nul le long de $I$.

Par conséquent, au voisinage d'un point de $I$, on peut prolonger $f$ par 0 
en dehors de $M$ et obtenir une fonction $\tilde f$ qui est $C^2$ et vérifie
$\Delta f=0$. Par unique prolongement des fonctions harmoniques, on a
$f\equiv0$ sur $M$.
\end{proof}
\subsection{Le problème de Steklov-Neumann}\label{rappels:sn}
 Étant donné une variété compacte à bord $M$, on se donne un domaine
(ou une union de domaines disjoints) à bord $C^1$ par morceaux de 
$\partial M$ qu'on notera 
$\partial M_S$, et on pose $\partial M_N=\partial M\backslash\partial M_S$.
Si $\rho$ est une fonction sur $\partial M_S$, le problème de 
Steklov-Neumann se pose ainsi :
\begin{equation}\label{rappels:eqsn}
\renewcommand{\arraystretch}{1.5}
\left\{\begin{array}{ll}
\Delta f=0 & \textrm{dans }M\\
\frac{\partial f}{\partial \nu}=\sigma \rho f & \textrm{sur }\partial M_S\\
\frac{\partial f}{\partial \nu}=0 & \textrm{sur }\partial M_N
\end{array}\right.
\renewcommand{\arraystretch}{1}
\end{equation}
c'est-à-dire qu'on demande à la fonction harmonique $f$ de vérifier la
condition de Neumann sur $\partial M_N$. On appellera respectivement
bord de Steklov et bord de Neumann les ensembles $\partial M_S$
et $\partial M_N$. Les solutions de ce problème interviennent dans 
l'étude du phénomène hydrodynamique de ballottement (\emph{sloshing
problem}): si on considère 
un fluide parfait incompressible contenu dans un récipient $M$ avec une 
surface libre $\partial M_S$, les petites oscillations périodiques du fluide 
correspondent aux solutions de (\ref{rappels:eqsn}) pour une fonction
$\rho$ constante (voir par exemple~\cite{kk01}).

Le problème de Steklov-Neumann possède un spectre discret et positif qu'on 
notera 
\begin{equation}
0=\sigma_0(M,\partial M_S,g,\rho)<\sigma_1(M,\partial M_S,g,\rho)\leq
\sigma_2(M,\partial M_S,g,\rho)\ldots
\end{equation}

Le spectre de Steklov-Neumann possède la même caractérisation
variationnelle que le spectre de Steklov, à condition
de restreindre l'intégrale sur le bord au bord de Steklov :
\begin{equation}\label{rappels:minmax2}
\sigma_k(M,\partial M_S,g,\rho)=\inf_{V_{k+1}\in H^1(M)}
\sup_{f\in V_{k+1}\backslash\{0\}}
\frac{\int_M|\de f|^2\de v_g}%
{\int_{\partial M_S}f^2\rho\ \de v_g},
\end{equation}
où $V_k$ parcours les sous-espaces de dimension~$k$ de $H^1(M)$.

L'opérateur Dirichlet-to-Neumann est bien défini sur $\partial M_S$
en considérant des fonctions harmoniques vérifiant la condition de 
Neumann sur $\partial M_N$ et vérifie toujours l'inégalité 
elliptique~(\ref{rappels:ell}).

On aura besoin du fait que si on se donne une fonction $f$ sur 
$\partial M_S$ et qu'on la prolonge en une fonction harmonique
(toujours notée $f$), sa norme $L^2$ sur $\partial M_N$ est 
contrôlée par sa norme
sur $\partial M_S$, c'est-à-dire qu'il existe une constante
$c>0$ ne dépendant que de $g$ et $\rho$ telle que $\int_{\partial M_N}f^2
\leq c\int_{\partial M_S}f^2\rho$. Cela découle du fait que la 
norme $L^2(\partial M_N)$ de $f$ est contrôlée
par sa norme $H^{1/2}(M)$, elle même contrôlée par sa
norme $L^2(\partial M_S)$ (cf.~\cite{ta96}, section~4.4).

On utilisera aussi un bref usage du spectre de Steklov-Dirichlet,
défini en considérant des fonctions harmoniques qui vérifient la
condition $f=0$ sur $\partial M\backslash\partial M_S$. La propriété
de ce spectre qui nous intéressera est qu'il est strictement positif 
(cf.~\cite{ag05}).

\subsection{Fonctions harmoniques et métriques singulières}%
\label{rappels:sing}
Dans la section suivante, on aura à manipuler des métriques discontinues.
Si $U$ est un domaine de $(M,g)$ et $\varepsilon\in]0,1[$ un réel fixé, 
elles seront de la forme
\begin{equation}
\left\{\begin{array}{ll}
g_\varepsilon=\varepsilon^2g & \textrm{ sur } U\\
g_\varepsilon=g & \textrm{ sur } M\backslash U.
\end{array}\right.
\end{equation}
Comme les normes $L^2$ et de Sobolev pour les métriques $g$ et 
$g_\varepsilon$ sont équivalentes, la théorie spectrale de la forme
quadratique $\|\de\cdot\|_{g_\varepsilon}^2$ sur $H^1(M)$ est
donc similaire à celle de $\|\de\cdot\|_g^2$. On peut donc définir 
le prolongement harmonique d'une fonction $f\in C^\infty(\partial M)$
comme étant le prolongement minimisant cette forme quadratique.
Cette définition est cohérente avec celle de la forme quadratique $Q$
donnée par~(\ref{rappels:Q}).

Un tel prolongement minimisera en particulier $\|\de\cdot\|_{g_\varepsilon}^2$
en restriction à chacun des domaines $U$ et $M\backslash U$, il sera donc 
harmonique au sens usuel sur ces deux domaines. 

 Les mêmes remarques s'appliquent au problème de Steklov-Neumann.

\section{Théorèmes de convergence spectrale}\label{conv}
\subsection{Rappels}
Dans cette section, nous allons montrer plusieurs théorèmes de convergence
spectrale dont nous aurons besoin pour prescrire le spectre de Steklov.
Nous utiliserons pour cela les techniques développées par Y.~Colin de Verdière
dans~\cite{cdv86}. Pour prescrire la multiplicité des valeurs propres,
il nous faudra montrer la convergence des espaces propres, et nous
aurons aussi besoin d'une certaine uniformité de la convergence, nous 
reprendrons pour cela les notations de \cite{cdv86} :

 Soit $E_0$ et $E_1$ sont deux sous-espaces vectoriels de même dimension~$N$
d'un espace de Hilbert, munis respectivement des formes quadratiques $q_0$
et $q_1$. Si $E_0$ et $E_1$ sont suffisamment proches, il existe
une isométrie naturelle $\psi$ entre les deux (voir la section~I de
\cite{cdv86} pour les détails de la construction), on définit alors
l'écart entre $q_0$ et $q_1$ par $\|q_1\circ\psi-q_0\|$. Pour deux
formes quadratiques $Q_0$ et $Q_1$ sur l'espace de Hilbert, on appellera
\emph{$N$-écart spectral entre $Q_0$ et $Q_1$} l'écart entre les
deux formes quadratiques restreintes à la somme des espaces propres associés
aux $N$ premières valeurs propres. Si cet écart est petit, alors les $N$
premières valeurs propres de $Q_0$ et leurs espaces propres sont proches
de ceux de $Q_1$.

 On veut montrer que la convergence spectrale est uniforme pour une certaine
famille de spectres limites. Comme dans \cite{cdv86} on dira donc qu'une forme
quadratique vérifie l'hypothèse ($*$) si ses valeurs propres vérifient
$$\lambda_1\leq\ldots\leq\lambda_N<\lambda_N+\eta\leq\lambda_{N+1}\leq M
\ \ (*)$$
pour un entier $N$ et des réels $\eta,M>0$ fixés une fois pour toute.

Comme dans \cite{cdv86}, on fera appel aux deux lemmes qui suivent. Les
constantes $N$, $M$ et $\eta$ qui interviennent dans les énoncés
font référence à l'hypothèse ($*$) définie plus haut.
\begin{lem}[\cite{cdv86}, th.~I.7]\label{conv:lem1}
Soit $Q$ une forme quadratique positive sur un espace de Hilbert $\mathcal H$
dont le domaine admet la décomposition $Q$-orthogonale
$\mathrm{dom}(Q)=\mathcal H_0\oplus\mathcal H_\infty$. Pour tout
$\varepsilon>0$, il existe une constante $C(\eta,M,N,\varepsilon)>0$ (grande)
telle que si $Q_0=Q_{|\mathcal H_0}$ vérifie l'hypothèse ($*$) et que
$\forall x\in\mathcal H_\infty,\ Q(x)\geq C|x|^2$,
alors $Q$ et $Q_0$ ont un $N$-écart spectral inférieur à $\varepsilon$.
\end{lem}

\begin{lem}[\cite{cdv86}, th.~I.8]\label{conv:lem2}
Soit $(\mathcal H,|\cdot|)$ un espace de Hilbert muni d'une forme quadratique
positive $Q$. On se donne en outre une suite de métriques $|\cdot|_n$
sur $\mathcal H$ et une suite de formes quadratiques $Q_n$ de même
domaine que $Q$ telles que:
\begin{itemize}
\item[(i)] il existe $C_1,C_2>0$ tels que $\forall x\in\mathcal H,\
C_1|x|\leq |x|_n\leq C_2|x|$;
\item[(ii)] pour tout $x\in \mathrm{dom}(Q)$, $|x|_n\to|x|$;
\item[(iii)] pour tout $x\in \mathrm{dom}(Q)$, $Q(x)\leq Q_n(x)$;
\item[(iv)] pour tout $x\in \mathrm{dom}(Q)$, $Q_n(x)\to Q(x)$.
\end{itemize}
Si $Q$ vérifie l'hypothèse ($*$), alors à partir d'un certain rang (dépendant
de $\eta$, $M$ et $N$), $Q$ et $Q_n$ ont un $N$-écart spectral inférieur à
$\varepsilon$.
\end{lem}
\begin{remarque}\label{conv:rem1}
Comme on l'a remarqué dans~\cite{ja11}, dans le lemme~\ref{conv:lem2}, 
on peut affaiblir l'hypothèse
$C_1|x|\leq |x|_n\leq C_2|x|$ en $C_1|x|\leq |x|_n\leq C_2|x|+
\varepsilon_nQ_n(x)^\frac12$ avec $\varepsilon_n\to0$,
la démonstration restant exactement la même (on peut aussi remplacer 
$Q_n$ par $Q$ dans cette dernière inégalité). En particulier, il n'est
pas nécessaire que l'espace de Hilbert $(\mathcal H,|\cdot|)$ soit
complet pour $|\cdot|_n$. 
\end{remarque}
\begin{remarque}\label{conv:rem2}
On peut aussi remplacer l'hypothèse $x\in\mathrm{dom}(Q)\Rightarrow
Q(x)\leq Q_n(x)$
par $Q(x)\leq M\Rightarrow Q(x)\leq (1+\varepsilon_n)Q_n(x)$ avec
$\varepsilon_n\to0$. 
\end{remarque}
\begin{remarque}
Pour déduire la convergence du spectre et des espaces propres de la
convergence des formes quadratiques, on doit en principe se ramener à une
norme de Hilbert fixe. Ça ne sera pas nécessaire dans la suite car les
étapes de la démonstration où la norme varie seront traitées à l'aide
du lemme~\ref{conv:lem2}.
\end{remarque}

\subsection{Densité et convergence de spectre}
Notre premier résultat de convergence sera de montrer qu'avec une densité
fixée sur le bord, on peut déformer conformément la métrique de manière à 
faire tendre le spectre de Steklov vers le spectre correspondant à une autre 
densité. On peut en outre faire tendre la métrique déformée vers la métrique
initiale dans l'intérieur de la variété. On se restreindra au cas où la 
densité initiale est plus petite que la densité du spectre limite, ce qui
sera suffisant pour les applications dans la section suivante.
\begin{theo}\label{conv:thdens}
Soit $(M,g)$ une variété riemannienne compacte à bord, et $\rho,\bar\rho\in
C^0(\partial M)$ deux fonctions sur le bord de $M$ telles que 
$\bar\rho\geq\rho$.

 Il existe une famille $g_\varepsilon$ de métriques conformes à $g$ sur $M$
telle que :
\begin{enumerate}
\item $\sigma_k(M,g_\varepsilon,\rho)$ tend vers $\sigma_k(M,g,\bar\rho)$ quand
$\varepsilon\to0$ pour tout $k\geq0$, avec convergence des espaces propres.
\item $g_\varepsilon=(\bar\rho/\rho)^{2/(n-1)}g$ sur $\partial M$.
\item $g_\varepsilon$ tend vers $g$ uniformément sur tout compact dans
l'intérieur de $M$.
\end{enumerate}
En outre, si les $\sigma_k(M,g,\bar\rho)$ vérifient l'hypothèse ($*$), 
alors le $N$-écart spectral entre $(M,g_\varepsilon,\rho)$
et $(M,g,\bar\rho)$ tend vers~0.
\end{theo}

\begin{proof}
On définit une famille $h_\varepsilon\in C^\infty(M)$ de facteurs conformes
de la manière suivante : on fixe $h_\varepsilon(x)=(\bar\rho/\rho)^{1/(n-1)}$
pour $x\in\partial M$ et on étend $h_\varepsilon$ de manière lisse de sorte
que la famille $(h_\varepsilon)$ tende simplement vers~1 dans l'intérieur
de $M$, et uniformément sur tout compact ne rencontrant pas le bord. On 
pose alors $g_\varepsilon=h_\varepsilon^2g$ pour tout $\varepsilon$.

La famille de métriques $g_\varepsilon$ induit les familles de normes et 
de formes quadratiques
\begin{equation}
Q_\varepsilon(f)=\inf_{\stackrel{\tilde f\in H^1(M)}{\tilde f_{|\partial M}=f}}
\int_Mh_\varepsilon^{n-2}|\de\tilde f|^2\de_{v_g}\textrm{ et } 
|f|_\varepsilon=\int_{\partial M}f^2\bar\rho\ \de_{v_g}.
\end{equation}
Comme $\bar\rho\geq\rho$, on peut choisir une suite $(h_\varepsilon)$ 
décroissante, les suites $Q_\varepsilon$ et $|\cdot|_\varepsilon$ vérifient 
alors les hypothèses du lemme~\ref{conv:lem2}, ce qui suffit pour conclure.
\end{proof}

\subsection{Convergence vers le spectre d'un domaine}
Le second théorème consiste à faire converger le spectre de Steklov 
d'une variété à bord $M$ vers le spectre de Steklov-Neumann d'un domaine
$U$ de $M$, avec la condition de Steklov sur $\partial U_S=
\partial U\cap\partial M$ 
et la condition de Neumann sur le reste du bord de $U$. Ce résultat 
étend au spectre de Steklov de théorèmes analogues concernant le laplacien
agissant sur les fonctions (\cite{cdv86}) et sur les formes différentielles
(\cite{ja11}).
\begin{theo}\label{conv:thdom}
Soit $(M^n,g)$ une variété riemannienne compacte à bord de dimension $n\geq3$, 
$\rho\in C^\infty(\partial M)$ et 
$U$ un domaine de $M$ à bord $C^1$ par morceau tel que 
$\partial U_S=\partial U\cap\partial M$ soit non vide.
Il existe une famille $g_\varepsilon$ de métriques sur $M$ conformes à $g$
telle que
\begin{enumerate}
\item $g=g_\varepsilon$ en restriction à $U$.
\item $\Vol(M,g_\varepsilon)\to\Vol(U,g)$ quand $\varepsilon\to0$.
\item $\sigma_k(M,\rho,g_\varepsilon)\to
\sigma_k(U,\partial U_S,\rho_{|\partial U_S},g_{|U})$ quand
$\varepsilon\to0$ pour tout $k\geq0$, avec convergence des espaces propres.
\end{enumerate}
En outre, si les $\sigma_k(U,\partial U_S,\rho_{|\partial U_S},g_{|U})$ 
vérifient l'hypothèse ($*$), alors le $N$-écart spectral entre 
$(M,\rho,g_\varepsilon)$
et $(U,\partial U_S,\rho_{|\partial U_S},g_{|U})$ tend vers~0.

\end{theo}

\begin{proof}
La démonstration est similaire à celle du théorème~III.1 de \cite{cdv86}
et passe par l'intermédiaire, pour un réel $\eta>0$ petit donné,
de la métrique singulière $g_\eta$ définie par $g_\eta=g$
sur $U$ et $g_\eta=\eta^2g$ sur $M\backslash U$. Elle se 
déroule en deux étapes : d'abord, on montre la convergence spectre pour
la famille de métriques singulières, puis on approche ces métriques 
singulières par des métriques lisses. On conclut en se donnant, pour
un $\varepsilon>0$ donné, une métrique $g_\eta$ tel que l'écart
spectral avec le spectre de $(U,g)$ soit inférieur à $\varepsilon$,
puis une métrique lisse $g_\varepsilon$ tel que l'écart spectral avec
$g_\eta$ soit lui aussi inférieur à $\varepsilon$.

On fera souvent appel à la forme quadratique définie en~(\ref{rappels:Q}),
en particulier quand on manipule des métriques singulières.

\emph{Étape 1.}
Un réel $\eta>0$ étant donné, la métrique $g_\eta$ induit sur
$L^2(\partial M)$ la forme
quadratique 
\begin{equation}
Q_\eta(f)=\inf_{\tilde f_{|\partial M}=f}
\left(\int_U|\de\tilde f|^2\de_{v_g}+\eta^{n-2}
\int_{M\backslash U} |\de\tilde f|^2\de_{v_g}\right)
\end{equation}
 et la norme 
$|f|_{g_\eta}=\int_{\partial U_S} f^2\rho\ \de_{v_g}+
\eta^{(n-1)}\int_{\partial M\backslash \partial U_S} f^2\rho\ \de_{v_g}$.
On va utiliser le lemme~\ref{conv:lem1} pour se ramener à un sous-domaine
de la forme quadratique puis appliquer le lemme~\ref{conv:lem2}.

En notant $\mathcal H$ le domaine de la forme quadratique $Q_\eta$,
on définit l'espace $\mathcal H_\infty=\{f\in \mathcal H, 
f_{|\partial U_S}=0\}$ et on note $\mathcal H_0$ son orthogonal pour la 
forme quadratique $Q_\eta$.
Pour appliquer le lemme~\ref{conv:lem1}, on doit minorer la forme quadratique
$Q_\eta$ sur $\mathcal H_\infty$ en fonction de $|\cdot|_\eta$.
Si $f\in\mathcal H_\infty$, alors $|f|_\eta^2=\eta^{(n-1)}
\int_{\partial M\backslash \partial U_S} f^2\rho\ \de_{v_g}
=\eta^{(n-1)}|f|^2$ et $Q_\eta(f)\geq\eta^{n-2}
\inf_{\tilde f_{|\partial M}=f}\int_M|\de\tilde f|^2\de_{v_g}=\eta^{n-2}Q(f)$.
On est donc ramené à l'étude du spectre de la forme quadratique $Q$
associée à la métrique initiale $g$ en restriction à l'espace 
$\mathcal H_\infty$, c'est-à-dire à minorer le spectre de Steklov sur $M$
avec condition de Dirichlet sur $\partial U_S$. Comme $0$ 
n'est pas dans le spectre de Steklov-Dirichlet (cf. 
paragraphe~\ref{rappels:sn}), il existe une constante 
$c>0$ telle que $Q(f)/|f|^2>c$ pour tout
$f\in\mathcal H_\infty$. Par conséquent, $Q_\eta(f)/|f|_\eta\geq
c\cdot\eta^{-2}$ pour tout $f\in\mathcal H_\infty$. Si $\eta$ est suffisamment
petit, on peut donc appliquer le lemme~\ref{conv:lem1} et en déduire que
le spectre pour la métrique $g_\eta$ est proche du spectre de $Q_\eta$ 
restreint à $\mathcal H_0$.

Il reste à montrer que la limite du spectre de ${Q_\eta}_{|\mathcal H_0}$
est le spectre de Steklov-Neumann du domaine $U$. On utilisera pour
cela le lemme~\ref{conv:lem2}.
Puisque $\mathcal H_0$ est défini comme le $Q_\eta$-orthogonal des fonctions de 
$\partial M$ nulles sur $\partial U_S$, une fonction de $\mathcal H_0$ est
entièrement déterminée par sa restriction à $\partial U_S$. Plus précisément, 
parmi les fonctions $f$ dont la valeur sur $\partial U_S$ est fixée, celle 
qui est dans $\mathcal H_0$ est celle minimise la forme quadratique $Q_\eta$. 
C'est donc la restriction au bord du prolongement harmonique (tel qu'on l'a
défini au paragraphe~\ref{rappels:sing}) de $f_{\partial U_S}$ avec 
condition de Neumann sur $\partial M\backslash\partial U_S$.
Dans la suite, on identifiera souvent une fonction sur $\partial U_S$ avec 
le prolongement ainsi défini.

 La norme $|\cdot|_\eta$ converge en décroissant vers la norme $|\cdot|$
définie par $|f|=\int_{\partial U_S}f^2\rho\ \de v_g$. L'hypothèse~(ii) et 
la première inégalité de l'hypothèse~(i) du lemme~\ref{conv:lem2} sont donc 
satisfaites. Les hypothèses~(iii) et~(iv) sont vérifiées pour les mêmes 
raisons.

Il reste à montrer que la deuxième inégalité de l'hypothèse~(i) est vérifiée. 
Pour cela, on doit majorer $\int_{\partial M\backslash\partial U_S}f^2\rho\ 
\de v_g$.
Notons $\tilde f$ le prolongement de $f_{|\partial U_S}$ qui est harmonique
au sens du paragraphe~\ref{rappels:sing}, c'est-à-dire que 
$Q_\eta(f)=\int_U|\de\tilde f|^2\de_{v_g}+\eta^{n-2}
\int_{M\backslash U} |\de\tilde f|^2\de_{v_g}$.
Comme on l'a remarqué au paragraphe~\ref{rappels:sn}, puisque $\tilde f$
est harmonique sur $M\backslash U$ avec condition de Neumann sur 
$\partial M\backslash\partial U_S$, la norme 
$L^2(\partial M\backslash\partial U_S,\rho)$ de $\tilde f$ est contrôlée par 
sa norme $L^2$ sur $\partial U\backslash\partial U_S$, c'est-à-dire que
$\int_{\partial M\backslash\partial U_S}\tilde f^2\rho\ \de v_g \leq 
c_1\int_{\partial U\backslash\partial U_S}\tilde f^2\rho\ \de v_g$.
Notons que la constante $c_1$ est invariante par homothétie, donc
indépendante de $\eta$, à condition de considérer sur $\partial(M\backslash U)$
la métrique induite par la métrique de $M\backslash U$. En considérant
la métrique $g_\eta$ sur $\partial M\backslash\partial U_S$ et la métrique $g$
sur $\partial U\backslash\partial U_S$ on obtient 
\begin{eqnarray}
\int_{\partial M\backslash\partial U_S}\tilde f^2\rho\ \de v_{g_\eta} & \leq &
\eta^{n-1}c_1\int_{\partial U\backslash\partial U_S}\tilde f^2\de v_g
\end{eqnarray}
On majore le membre de droite à l'aide de l'inégalité elliptique de
l'opérateur Dirichlet-to-Neumann sur $\partial U$.
\begin{eqnarray}
\int_{\partial U\backslash\partial U_S}\tilde f^2\de v_g &\leq&
\|f\|^2_{L^2(\partial U)}\leq \|f\|^2_{H^1(\partial U)}\nonumber\\
&\leq & c_2\int_{\partial U}f\frac{\partial f}{\partial\nu}\de v_g
=c_2\int_U|\de\tilde f|^2\de v_g\nonumber\\
&\leq &c_2Q_\eta(f).
\end{eqnarray}
 On a donc finalement $|f|_\eta^2\leq|f|^2+\eta^{n-1}c_1c_2Q_\eta(f)$ ce
qui permet d'appliquer le lemme~\ref{conv:lem2} et la 
remarque~\ref{conv:rem1}.

\emph{Étape 2.}
On doit montrer que pour tout $\eta>0$, le spectre de $Q_\eta$
peut être approché par le spectre de métriques lisses conformes à $g$.

Le paramètre $\eta$ état fixé, on définit une suite de facteurs conformes 
$h_i$ tels que la suite $(h_i)$ converge en décroissant vers la fonction
$\chi_U+\eta\chi_{M\backslash U}$ et on pose $g_i=h_i^2g$. Les suites de
norme de Hilbert $|\cdot|_i$ et de formes quadratiques $Q_i$ associées à
$g_i$ convergent vers $|\cdot|_\eta$ et $Q_\eta$ en vérifiant les hypothèses
du lemme~\ref{conv:lem2}, ce qui assure la convergence du spectre et
des espaces propres.
\end{proof}

\subsection{Convergence vers le spectre du bord}
Enfin, nous allons montrer qu'on peut faire tendre le spectre de Steklov
homogène (c'est-à-dire que $\rho\equiv1$)
d'une variété à bord $M$ vers le spectre du laplacien sur $\partial M$, 
la métrique sur $\partial M$ restant homothétique à la métrique initiale.
Bien que ce théorème ne soit pas indispensable pour démontrer le 
théorème~\ref{intro:th1}, on peut l'utiliser si la dimension de~$\partial M$
est au moins~3. Il semble aussi intéressant en lui-même et fournit un exemple
d'application du théorème démontré au paragraphe précédent.
\begin{theo}\label{conv:bord}
Soit $(M^n,g)$ une variété riemannienne compacte à bord de dimension 
$n\geq3$. Il existe une
famille  $g_\varepsilon$ de métriques sur $M$ conformes à $g$ et homothétiques
à $g$ le long de $\partial M$ telle que
 $\sigma_k(M,g_\varepsilon)\to\lambda_k(M,g)$ quand 
$\varepsilon\to0$ pour tout $k\geq0$.

En outre, si les $\lambda_k(M,g)$ vérifient l'hypothèse ($*$),
alors le $N$-écart spectral entre le spectre de Steklov de $(M,
g_\varepsilon)$ et le spectre du laplacien de $(\partial M,g)$
tend vers~0.
\end{theo}

\begin{proof}
Le principe de la démonstration consiste à se ramener au cas d'un
voisinage collier du bord (avec la condition mixte Steklov-Neumann) en
utilisant le théorème~\ref{conv:thdom}. On va procéder en trois étapes :
d'abord montrer la convergence du spectre d'un voisinage collier de 
$\partial M$ muni d'une métrique produit, puis traiter le cas de la 
restriction de métrique $g$ à ce voisinage collier, et enfin montrer la 
convergence du spectre de $M$.

\emph{Étape 1.} On va déterminer l'asymptotique (quand $\eta\to0$) du 
spectre de la variété
produit $\partial M\times [0,\eta]$ pour un métrique produit, 
avec la condition de Steklov sur  $\partial M\times\{0\}$ et la 
condition de Neumann sur $\partial M\times\{\eta\}$.

On peut déduire ce spectre du spectre de Steklov de $\partial M\times 
[0,2\eta]$ (avec condition de Steklov sur les deux bords) par symétrie: 
en effet, on peut partitionner les valeurs propres de $\partial M\times
[0,2\eta]$ en deux, selon que les fonctions propres sont symétriques
ou antisymétriques. Ces fonctions propres vérifient la condition
de Neumann sur $\partial M\times\{\eta\}$ dans le premier cas et
la condition de Dirichlet dans le second cas. Le spectre de $\partial M
\times [0,\eta]$ avec condition mixte est donc le spectre de Steklov
de $\partial M\times [0,2\eta]$ restreint aux fonctions symétriques.

Le spectre de Steklov de $\partial M\times [0,2\eta]$ a été calculé 
explicitement dans~\cite{cesg11} (lemme~6.1) en fonction du spectre du
laplacien sur $\partial M$ : si $\partial M$ est de volume~1 et si
on note $\lambda_k$ ses valeurs propres et $u_k$ ses fonctions propres,
alors le spectre non nul de $\partial M\times [0,2\eta]$ restreint 
aux fonctions symétriques est $\sqrt{\lambda_k}
\tanh(\eta\sqrt{\lambda_k})$ les fonctions propres associées étant
$\cosh(\sqrt{\lambda_k}t)u_k(x)$, où $x$ désigne un point de $\partial M$
et $t$ la coordonnée sur l'intervalle. Le spectre de Steklov-Neumann de 
$\partial M\times [0,\eta]$ se comporte donc asymptotiquement 
comme $\eta\lambda_k$ quand $\eta\to0$. En pratiquant une homothétie sur 
$\partial M\times [0,\eta]$, on peut donc faire tendre son
spectre vers $\lambda_k$. On peut facilement vérifier à l'aide de l'expression
des fonctions propres qu'il y a bien convergence des espaces propres.

\emph{Étape 2.}
Étant donnée la variété à bord $(M,g)$ et un réel $\eta>0$ petit, on
considère le $\eta$-voisinage collier $M_\eta$ de $\partial M$,
c'est-à-dire que $M_\eta=\{x\in M, d(x,\partial M)\leq\eta\}$.
Pour $\eta$ suffisamment petit, $M_\eta$ est difféomorphe au
produit de $\partial M$ avec un intervalle. On considère alors le problème
de Steklov-Neumann sur $M_\eta$ comme dans l'étape~1.

Quand $\eta$ tend vers~0, la métrique $g$ restreinte à $M_\eta$
est de plus en plus proche d'une métrique produit. Plus précisément, il 
existe une famille de réels $\tau_\eta>1$ telle que $\tau_\eta
\to1$ quand $\eta\to0$ et $1/\tau_\eta g_\eta\leq
g_{|M_\eta}\leq\tau_\eta g_\eta$, où $g_\eta$
désigne la métrique produit sur $\partial M\times[0,\eta]$.
Comme la constante $\tau_\eta$ contrôle aussi l'écart entre les
normes de Hilbert et les formes quadratiques pour le problème de 
Steklov-Neumann sur $(M_\eta,g)$ et $(M_\eta,g_\eta)$,
on peut appliquer le lemme~\ref{conv:lem2} et la remarque~\ref{conv:rem2}
pour obtenir la convergence spectrale comme dans l'étape~1.

\emph{Étape 3.}
Pour pouvoir conclure, il on aura besoin de faire
tendre le spectre de Steklov de $M$ vers celui de $M_\eta$
en restant dans la classe conforme de $g$. On va utilise pour cela
le théorème~\ref{conv:thdom}:

Un réel $\varepsilon>0$ étant donné, on choisit
$\eta>0$ et un rapport d'homothétie $r_\eta>0$ tels que le spectre de 
$(M_\eta,r_\eta^2g)$ soit $\varepsilon$-proche de celui de $(\partial M,g)$.
Puis, en appliquant le théorème~\ref{conv:thdom} avec $U=M_\eta$, 
on obtient une métrique 
$g'_\eta$ sur $M$ tel que le spectre de $(M,g'_\eta)$ soit 
$\varepsilon$-proche de celui de $(M_\eta,r_\eta^2g)$.
 Quand $\varepsilon$ tend vers~0, on a ainsi
convergence du spectre et des espaces propres de $(M,\rho,g_\varepsilon)$
vers ceux du laplacien sur $(\partial M,g)$.
\end{proof}

\section{Prescription du spectre}\label{presc}
\subsection{L'hypothèse de transversalité d'Arnol'd}
Pour prescrire la multiplicité des valeurs propres de Steklov
nous utilisons, selon la méthode introduite par Colin de Verdière, trois
ingrédients : les théorèmes de convergence spectrale démontrés dans
les sections précédentes, des modèles de valeurs propres multiples déjà
connus et une propriété de stabilité vérifiée par ces modèles. Nous 
allons commencer par rappeler cette dernière. On verra au paragraphe
suivant comment utiliser des graphes complets comme modèles de spectre 
avec multiplicité.

On suppose qu'on a une famille d'opérateurs $(P_a)_{a\in B^k}$, où $B^k$
est la boule unité de $\R^k$ (en pratique, $P_a$ est l'opérateur 
Dirichlet-to-Neumann associé à une métrique $g_a$), tels que
$P_0$ possède une valeur propre $\lambda_0$ d'espace propre $E_0$ et
de multiplicité $N$. Pour les petites valeurs de $a$, $P_a$ possède des
valeurs propres proches de $\lambda_0$ dont la somme des espaces propres
est de dimension~$N$. Comme dans la définition de l'écart spectral,
on identifie cette somme à $E_0$ et on note $q_a$ la forme
quadratique associée à $P_a$ transportée sur $E_0$.
\begin{definition}[\cite{cdv88}]\label{presc:def}
On dit que $\lambda_0$ vérifie l'hypothèse de transversalité d'Arnol'd si
l'application $\Psi:a\mapsto q_a$ de $B^k$ dans $\mathcal Q(E_0)$ est
essentielle en $0$, c'est-à-dire qu'il existe $\varepsilon>0$ tel que si
$\Phi:B^k\to\mathcal Q(E_0)$ vérifie $\|\Psi-\Phi\|_{\infty}\leq\varepsilon$,
alors il existe $a_0\in B^k$ tel que $\Phi(a_0)=q_0$.
\end{definition}
Une propriété cruciale est que si $\Phi$ provient d'une famille $(P'_a)$
d'opérateurs, alors $\lambda_0$ est valeur propre de $P'_{a_0}$ de
multiplicité~$N$ et vérifie la même propriété de transversalité, ce
qui justifie qu'on parle de stabilité de la multiplicité.  Comme remarqué 
dans \cite{cdv88}, on peut généraliser cette définition à une suite finie 
de valeurs propres.

\subsection{Voisinages tubulaires de graphes}\label{presc:graphe}
Dans \cite{cdv88}, Y. Colin de Verdière a montré qu'un graphe complet
muni d'un laplacien combinatoire et d'une métrique appropriée possède une 
(ou plusieurs) valeur propre multiple vérifiant
la propriété de transversalité d'Arnol'd. Dans ce paragraphe, nous allons
utiliser ce résultat pour construire une variété dont on prescrit le 
début du spectre Steklov-Neumann avec multiplicité.

On note $\Gamma$ un graphe fini, $S$ l'ensemble de ses sommets et 
$A$ l'ensemble de ses arêtes. On se donne
une métrique sur ce graphe en associant à chaque arête $a_i\in A$, une 
longueur $l_i>0$. Le laplacien combinatoire sur $\Gamma$ est l'opérateur
agissant sur les fonctions $S\to\R$ induit par la forme quadratique
$q(f)=\sum_{a_i\in A} l_i\de_{a_i}(f)^2$, avec $\de_a(f)=(f(x)-f(y))/l_i$,
$x$ et $y$ étant les extrémités de l'arête $a_i$. L'espace des fonctions
$\R^S$ sur les sommets est muni de sa structure euclidienne canonique.

On utilise alors le résultat de Colin de Verdière suivant:
\begin{theo}[\cite{cdv88}, section~4]\label{presc:thcdv}
Étant donné une suite $0=\lambda_0<\lambda_1\leq\lambda_2\leq\cdots\leq
\lambda_N$, il existe
une métrique sur le graphe complet à $N+1$ sommets dont le spectre
est la suite $(\lambda_i)$. De plus, ce spectre vérifie la propriété de 
transversalité d'Arnol'd.
\end{theo}

Il reste à construire une famille de variétés dont le début du spectre 
ressemble à celui d'un graphe complet. Ces variétés, qu'on notera
$\Omega_\varepsilon$, seront localement des domaines euclidiens.
Le graphe $\Gamma$ sera plongeable isométriquement dans $\Omega_\varepsilon$,
les sommets étant situés sur le bord de Steklov de la variété (construire
$\Omega_\varepsilon$ comme étant globalement un domaine euclidien 
nécessiterait d'imposer des contraintes sur les longueurs $l_i$,
ce qu'on veut éviter).

Plus précisément, deux réels $c>0$ et $\varepsilon$ étant fixé, on se donne 
pour chaque sommet $s_j\in S$ une demi-boule $B_j$ de rayon $c\varepsilon$
(dans la suite, utilisera l'indice $j$ pour les sommets du graphe et
on réservera l'indice $i$ pour les arêtes).
La constante $c$ sera fixée plus loin. Pour chaque arête 
$a_i\in A$, on se donne ensuite un cylindre $C_i$ de rayon $\varepsilon$ et 
de longueur $l_i-2c\varepsilon$. Si on note $j$ et $j'$ les indices des 
sommets extrémités de l'arête $a_i$, et qu'on place les boules $B^{n-1}$ 
qui bordent le cylindre $C_i$ tangentiellement aux demi-boules $B_j$ et 
$B_{j'}$, on peut plonger isométriquement l'arête $a_i$ dans la réunion
de $C_i$, $B_j$ et $B_{j'}$,en identifiant les sommets $s_j$ et $s_{j'}$ 
avec les centres de $B_j$ et $B_{j'}$. En répétant le procédé pour chaque 
arête, le graphe $\Gamma$ se plonge isométriquement dans la réunion des 
demi-boules $B_j$ et des cylindres $C_i$. Pour
construire le domaine $\Omega_\varepsilon$, on va prolonger le cylindre
$C_i$ en un cylindre $C'_i$ dont les extrémités seront dans les demi-boules.
Chaque boule $B_j$ étant vue localement comme un domaine euclidien, on 
construit une application de $C'_i=B^{n-1}\times[0,l_i]$ dans la réunion
des $C_i$ et des voisinages des $B_j$ telle que:
\begin{itemize}
\item $B^{n-1}\times[c\varepsilon,l_i-c\varepsilon]$ est envoyé isométriquement
sur $C_i$;
\item pour $t\in[0,c\varepsilon]$ et $[l_i-c\varepsilon,l_i]$, chaque boule 
$B^{n-1}\times\{t\}$ est plongée isométriquement dans
le voisinage de la demi-boule $B_j$ correspondante; 
\item $B^{n-1}\times\{0\}$ et $B^{n-1}\times\{1\}$ sont envoyé sur le 
bord équatorial de $B_j$;
\item l'application obtenue est $1$-lipschitzienne.
\end{itemize} 
Les extrémités des cylindres ne sont donc pas isométriques à la métrique
produit mais légèrement tordus à l'intérieur des boules (cf. 
figure~\ref{presc:fig}). En outre, on
fait en sorte que les images de chaque $C'_i$ soient disjointes 
(on choisit $c$ assez grand pour que ça soit possible). Ces précisions
techniques faciliteront l'étude du spectre.

\begin{figure}[h]
\begin{center}
\begin{picture}(0,0)%
\includegraphics{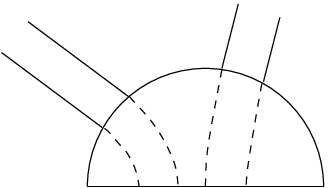}%
\end{picture}%
\setlength{\unitlength}{4144sp}%
\begingroup\makeatletter\ifx\SetFigFont\undefined%
\gdef\SetFigFont#1#2#3#4#5{%
  \reset@font\fontsize{#1}{#2pt}%
  \fontfamily{#3}\fontseries{#4}\fontshape{#5}%
  \selectfont}%
\fi\endgroup%
\begin{picture}(2500,1414)(686,-973)
\put(2841,214){\makebox(0,0)[lb]{\smash{{\SetFigFont{12}{14.4}{\rmdefault}{\mddefault}{\updefault}{\color[rgb]{0,0,0}$C_i$}%
}}}}
\put(3171,-691){\makebox(0,0)[lb]{\smash{{\SetFigFont{12}{14.4}{\rmdefault}{\mddefault}{\updefault}{\color[rgb]{0,0,0}$B_j$}%
}}}}
\put(2341,-646){\makebox(0,0)[lb]{\smash{{\SetFigFont{12}{14.4}{\rmdefault}{\mddefault}{\updefault}{\color[rgb]{0,0,0}$C'_i$}%
}}}}
\end{picture}%
\end{center}
\caption{Construction locale de $\Omega_\varepsilon$%
\label{presc:fig}}
\end{figure}
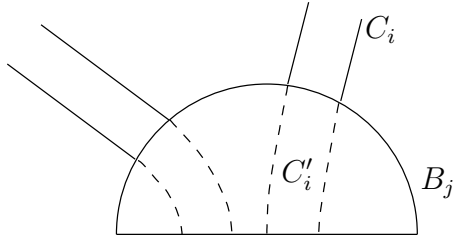

On peut remarquer que quand $\varepsilon$ tend vers~0, $\Omega_\varepsilon$ 
tend vers le graphe~$\Gamma$ pour la distance de Gromov-Hausdorff.

 En considérant la condition de Steklov sur les boules équatoriales
des demi-boules $B_j$ et la condition de Neumann sur le reste du bord 
de $\Omega_\varepsilon$, on va montrer que le début du spectre de 
$\Omega_\varepsilon$ tend vers le spectre du graphe, à une constante
multiplicative près:
\begin{theo}\label{conv:graphe}
Le $N$-écart spectral entre le spectre de Steklov-Neumann de 
$\Omega_\varepsilon$ et le spectre de $c^{n-1}\Delta_\Gamma$, où
$\Delta_\Gamma$ désigne le laplacien combinatoire sur $\Gamma$, tend vers~$0$
quand~$\varepsilon$ tend vers~$0$.
\end{theo}
\begin{proof}
La démonstration se déroule en deux étapes. D'abord, on décompose l'espace
des fonctions harmoniques (pour le problème de Steklov-Neumann) 
sur $\Omega_\varepsilon$ en deux sous-espaces pour
appliquer le lemme~\ref{conv:lem1}, puis on montre la convergence en
se restreignant à l'un des sous-espaces. On note $\partial
\Omega_{\varepsilon,S}$ le bord de Steklov de $\Omega_\varepsilon$, 
$\partial\Omega_{\varepsilon,S}^j$, $1\leq j\leq N$ ses composantes connexes
et $\mathcal H$ l'espace des fonctions harmoniques sur $\Omega_\varepsilon$
vérifiant la condition de Neumann sur $\partial\Omega_\varepsilon\backslash 
\partial\Omega_{\varepsilon,S}$.

\emph{Étape 1.} On définit l'espace $\mathcal H_0$ comme étant l'espace
des fonctions harmoniques de $\Omega_\varepsilon$ constantes sur
chacune des $N$ composantes connexes du bord de Steklov. L''orthogonal 
de $\mathcal H_0$ pour la forme quadratique $Q$ associée à l'opérateur 
$\Lambda$ contient les fonction constantes, qui sont aussi dans $\mathcal H_0$.
On définit donc $\mathcal H_\infty$, comme l'espace des fonctions 
$Q$-orthogonales à $\mathcal H_0$ et d'intégrale nulle sur 
$\partial\Omega_{\varepsilon,S}$. Si $f\in\mathcal H_0$ et si $g$ est 
$Q$-orthogonale à $H_0$, alors on a, 
en notant $f_j$ la valeur de $f$ sur $\partial\Omega_{\varepsilon,S}^j$: 
$(f,\Lambda g)=\int_{\partial\Omega_{\varepsilon,S}} 
f\frac{\partial g}{\partial\nu}=\sum_jf_j
\int_{\partial\Omega_{\varepsilon,S}^j}\frac{\partial g}{\partial\nu}$. Comme 
$(f,\Lambda g)$ est nul pour tout $f\in\mathcal H_0$, on en déduit que
$\int_{\partial\Omega_{\varepsilon,S}^j}\frac{\partial g}{\partial\nu}=0$ 
pour tout $j$. On a donc
\begin{equation}
\mathcal H_\infty=\left\{f\in\mathcal H,\ \int_{\partial\Omega_{\varepsilon,S}}
f=0,\ \int_{\partial\Omega_{\varepsilon,S}^j}\frac{\partial f}{\partial\nu}
=0\textrm{ pour tout }j\right\}.
\end{equation}

On doit minorer la forme quadratique $Q$ sur l'espace $\mathcal H_\infty$.
Pour cela, on va passer par l'intermédiaire du domaine $D=\cup_j B_j$.
Mais comme la restriction des formes harmoniques de $\Omega_\varepsilon$
ne vérifient pas la condition de Neumann sur les hémisphères qui bordent
les $B_j$ on va d'abord reformuler l'expression du bas du spectre de $Q$ sur 
$\mathcal H_\infty$. On pose, pour toute fonction $\tilde f\in C^\infty(F)$
telle que $\int_{\partial\Omega_{\varepsilon,S}}\tilde f=0$,
\begin{equation}
\tilde Q(\tilde f)=\inf_{\stackrel{f_{|\partial\Omega_{\varepsilon,S}}=\tilde f}
{\int_{\partial\Omega_{\varepsilon,S}^j}\frac{\partial f}{\partial\nu}=0}}
\int_{\Omega_\varepsilon}|\de f|^2.
\end{equation}
On peut vérifier que la borne inférieure de $\tilde Q$ 
(pour $\|\tilde f\|_2=1$) coïncide bien avec le bas du spectre de $Q$ sur 
$\mathcal H_\infty$.

On définit les espaces $\mathcal H^D$, $\mathcal H_\infty^D$ et les
formes quadratiques $Q^D$, $\tilde Q^D$ en remplaçant $\Omega_\varepsilon$ par
$D$ dans les définitions de $\mathcal H$,  $\mathcal H_\infty$, $Q$ et
$\tilde Q$. Comme $D\subset \Omega_\varepsilon$, $\tilde Q$ est minoré par 
la forme quadratique $\tilde Q^D$. La première valeur propre de 
$Q_{\mathcal H_\infty}$ est donc minorée par la première valeur propre de
$Q_{\mathcal H_\infty}^D$. Le domaine $D$ possède $N$ composantes connexes, 
donc la multiplicité de~0 dans spectre de Steklov de $D$ est $N$, les
fonctions propres étant les fonctions constantes sur chaque $B_j$.
La première valeur propre de $Q_{\mathcal H_\infty}^D$ est donc la 
$(N+1)$-ième valeur propre de $D$, qui est la première valeur propre non 
nulle $\sigma_1(B(\varepsilon))$ de la demi-boule 
de rayon $c\varepsilon$. Cette valeur propre se comporte comme 
$\varepsilon^{-2}$ quand $\varepsilon\to0$, ce qui permet d'appliquer 
que lemme~\ref{conv:lem1}.

\emph{Étape~2.} On doit maintenant comparer les spectres de $Q_{\mathcal H_0}$
et de la forme quadratique $q$ associée au laplacien combinatoire sur 
$\Gamma$. Les deux domaines des formes quadratiques sont en bijection 
de manière évidente, en identifiant une fonction sur les sommets $s_j$ 
de $\Gamma$ avec une fonction constante sur chaque $F_j$, prenant les 
mêmes valeurs. Les normes sur les deux espaces sont différentes. On notera
$|\cdot|_\Gamma$ la norme euclidienne canonique sur $\R^S$, la norme sur 
$\mathcal H_0$ est alors 
$|\cdot|=(c\varepsilon)^{n-1}\omega_{n-1}|\cdot|_\Gamma$, où $\omega_{n-1}$
désigne le volume de la boule euclidienne canonique de dimension~$n-1$.

Étant donnée  $f$ est une fonction sur $S$, on construit une fonction
test $\bar f$ sur $\Omega_\varepsilon$ prenant les mêmes valeurs que $f$
sur chaque $F_j$, constante sur chaque demi-boule $B_j$ et prolongée
de manière affine sur les cylindres $C_i$ constituant le domaine
$\Omega_\varepsilon$. On a $Q(f)\leq Q(\bar f)=
\int_{\Omega_\varepsilon}|\de\bar f|^2=\varepsilon^{n-1}\omega_{n-1}
\sum_{a_i\in A}\frac1{l_i-2c\varepsilon}
(f(x_i)-f(y_i))^2$, où $x_i$ et $x_i$ sont les extrémités de l'arête $a_i$,
donc $\limsup_{\varepsilon\to0} Q(f)/|f|^2\leq c^{n-1}q(f)/|f|_\Gamma^2$.

Réciproquement, étant donné une fonction $f\in\mathcal H_0$, on 
construit une fonction test sur les arêtes du graphe $\Gamma$.
À partir de la donnée de $f$ sur l'image d'un cylindre $C'_i=B^{n-1}\times
[0,\varepsilon,l_i]$, on définit
une fonction $\tilde f$ sur l'intervalle $[0,l_i]$ 
par moyennation sur chaque boule $B^{n-1}$, c'est-à-dire que
$\tilde f(t)=\frac1{\varepsilon^{n-1}\omega_{n-1}}\int_{B^{n-1}}
f(x,t)\de x$. On a alors, en utilisant le fait que plongement 
de $C'_i$ dans $\Omega_\varepsilon$ est $1$-lipschitzien,
\begin{equation}
|\de\tilde f|^2\leq 
\frac1{\varepsilon^{n-1}\omega_{n-1}}\int_{B^{n-1}}|\de f|^2\de x.
\end{equation}
On obtient ainsi une fonction $\tilde f$ sur $\Gamma$ qui est $C^1$,
qui coïncide avec $f$ sur les sommets  
et qui vérifie $\varepsilon^{n-1}\omega_{n-1}\|\de\tilde f\|^2\leq
\int_{\Omega_\varepsilon}|\de f|^2=Q(f)$. Comme sur le graphe, on a
$\|\de\tilde f\|^2\geq q(f)$, on obtient que
$\liminf_{\varepsilon\to0} Q(f)/|f|^2\geq c^{n-1}q(f)/|f|_\Gamma^2$.

On a finalement montré que $Q(f)/|f|^2$ converge simplement vers
$c^{n-1}q(f)/|f|_\Gamma^2$ quand $\varepsilon\to0$. Comme on travaille
sur des espaces de dimension finie, cela suffit pour assurer la convergence
du spectre et des espaces propres des deux opérateurs.
\end{proof}

\subsection{Application à la prescription de spectre}
On a maintenant tous les ingrédients pour montrer le théorème~\ref{intro:th1}.
La méthode la plus directe serait d'utiliser le théorème~\ref{conv:bord} 
de convergence du spectre vers celui du bord et les résultats de prescription
obtenus par Colin de~Verdière dans~\cite{cdv87} (on peut les adapter de manière
à prescrire la classe conforme). Cependant, cette méthode ne fonctionne
que si la dimension du bord est au moins~3. On va donc procéder autrement
en utilisant les plongements de graphes construits au paragraphe précédent.

\begin{proof}[du théorème~\ref{intro:th1}]
D'après le théorème~\ref{presc:thcdv}, il existe un graphe complet 
$\Gamma$ ayant le spectre voulu, avec la propriété de stabilité. On va 
transplanter ce spectre dans la variété $M$ en commençant par traiter 
le cas $\rho=1$.

On commence par déformer la variété $M$ en respectant la classe conforme et 
de manière à pouvoir plonger isométriquement le graphe $\Gamma$ dans 
$M$ en plaçant les sommets sur $\partial M$. Comme la dimension de $M$ 
est plus grande que~3, on peut le faire sans que les arêtes se croisent.
On note $g$ la métrique obtenue sur $M$.

\emph{A priori}, la métrique au voisinage du plongement de $\Gamma$
n'est pas euclidienne, on ne peut donc pas plonger isométriquement
un ouvert $\Omega_\varepsilon$ (construit au paragraphe précédent) au
voisinage de $\Gamma$. Cependant, pour tout $\varepsilon$ on peut 
déformer (de manière non conforme) la métrique $g$ en une métrique 
$g_\varepsilon$ telle que les graphes soient toujours plongés 
isométriquement et que $\Gamma$ possède un voisinage isométrique 
au domaine $\Omega_\varepsilon$. On peut de plus faire en sorte que
$\frac1{\tau_\varepsilon}g_\varepsilon\leq g\leq \tau_\varepsilon
g_\varepsilon$, pour une famille de réels $\tau_\varepsilon>1$ telle que
$\tau_\varepsilon\to1$ quand $\varepsilon\to0$.

On peut maintenant appliquer les résultats de convergence spectrale de
la section précédente. Pour un $\delta>0$ petit donné, on peut trouver
un $\varepsilon$ tel que le $N$-écart spectral entre $\Gamma$
et $\Omega_\varepsilon$ soit inférieur à $\delta$. En utilisant les
arguments de la démonstration du théorème~\ref{conv:bord} (étape~2),
on peut choisir $\varepsilon$ suffisamment petit pour que le $N$-écart
spectral entre $(\Omega_\varepsilon,g_\varepsilon)$ et
$(\Omega_\varepsilon,g)$ soit inférieur à $\delta$. Enfin, on peut faire
converger le spectre de $M$ vers celui de $(\Omega_\varepsilon,g)$
d'après le théorème~\ref{conv:thdom}, et en particulier déformer $g$
de manière conforme de sorte que le $N$-écart spectral entre $M$ et 
$(\Omega_\varepsilon,g)$ soit lui aussi inférieur à $\delta$. 
On peut donc rentre le $N$-écart spectral entre $M$ et $\Gamma$ 
arbitrairement petit, et ce de manière conforme.

Traitons maintenant le cas où $\rho$ varie. Quitte à multiplier les
$a_i$ par une constante, on peut supposer que $\rho\leq1$. Il suffit
d'ajouter une étape à la construction précédente et d'utiliser le 
théorème~\ref{conv:thdens} pour faire tendre le spectre
de la variété $(M,\rho)$ vers celui de $M$ sans densité.
\end{proof}

\section{Multiplicité en dimension~2}\label{bornes}
\subsection{Lignes nodales des fonctions propres}
On va montrer dans cette section les obstructions à la prescription
de multiplicité en dimension~2 (théorème~\ref{intro:th2}, \ref{intro:th3}
et~\ref{intro:th4}).
Comme dans le cas du laplacien, les deux principaux ingrédients sont 
le théorème nodal de Courant et le théorème de Cheng sur la structure
local de l'ensemble nodal. Dans toute la suite du texte, les fonctions
propres considérées seront les fonctions harmoniques sur $M$ et pas leur
restriction à $\partial M$. En particulier, les lignes et les domaines 
nodaux seront considérés sur $M$.

Avec ces précisions, le théorème nodal de Courant est valide pour
les problèmes de Steklov et Steklov-Neumann, quelle que soit la dimension:
\begin{theo}
Le nombre de domaines nodaux de la $k$-ième fonction propre du problème
de Steklov (ou de Steklov-Neumann) est au plus
égal à $k+1$.
\end{theo}
La démonstration (essentiellement la même que dans le cas du laplacien) 
est donnée dans \cite{ks69} pour la dimension~2, et elle se généralise
immédiatement en toute dimension.

Contrairement aux fonctions propres du laplacien, les fonctions propres
de Steklov ont la particularité que leurs domaines nodaux rencontrent
toujours le bord. Cette propriété a déjà été utilisée, par exemple dans
\cite{bkps10}, et nous y feront appel pour démontrer le 
théorème~\ref{intro:th3}:
\begin{lem}\label{nodal:lem}
Tout domaine nodal rencontre le bord de la variété. Dans le cas
du problème de Steklov-Neumann, tout domaine nodal rencontre le bord
de Steklov.
\end{lem}
\begin{proof}
Soit $f$ une fonction harmonique non nulle et $D$ un domaine nodal de $f$ 
ne rencontrant pas le bord de Steklov de la variété. Comme $f$ est 
harmonique et nulle sur le bord de $D$ (ou vérifie la condition de Neumann
le long du bord de Neumann de la variété), elle est uniformément nulle 
dans~$D$. Par conséquent, elle est nulle partout.
\end{proof}

Dans~\cite{ch76}, S.~Y.~Cheng a décrit la structure locale de l'ensemble
nodal des fonctions propres du laplacien en dimension~2. On peut les 
appliquer aux fonctions harmoniques, et le lemme qui précède permet 
de préciser certaines propriétés topologiques des domaines et des lignes 
nodales, en particulier leur incompressibilité (Une partie d'une surface 
est dite incompressible si son groupe fondamental s'injecte dans celui
de la surface). L'énoncé qui suit rassemble ces résultats:
\begin{theo}\label{nodal:th}
Supposons que $M$ est de dimension~2. Si $f$ est une fonction propre
du problème de Steklov, alors :
\begin{enumerate}
\item Les domaines nodaux de $f$ sont incompressibles.
\item L'ensemble nodal de $f$ intérieur à $M$ est la réunion d'un
nombre fini de courbes $C^2$ qui sont soit des cercles immergés, soit des arcs 
immergés dont les extrémités sont sur $\partial M$.
\item La réunion de ces courbes forme un graphe fini dont les composantes
connexes sont incompressibles.
\item Soit $p$ un point intérieur à $M$. Si $p$ est un point critique de 
$f$ situé sur l'ensemble nodal et que 
l'ordre d'annulation de $f$ en $p$ est $k$, alors au voisinage de $p$
l'ensemble nodal est la réunion de $k$ courbes s'intersectant en $p$, de
courbure géodésique nulle en $p$ et formant un système équiangulaire 
(en particulier, les sommets du graphe nodal intérieur à $M$ sont 
de degré pair).
\item
Tout point du bord où $f$ s'annule est l'extrémité d'une ligne nodale
intérieure à $M$.

\item Chaque composante connexe du bord contient un nombre pair d'extrémités
du graphe nodal.
\item  Dans le cas du problème de Steklov-Neumann, si $p$ est un
point du bord de Neumann où $f$ s'annule, l'ordre d'annulation $k$ de 
$f$ en $p$ est fini et le point $p$ est un zéro isolé en restriction à 
$\partial M_N$. Au voisinage de $p$ dans $M$, l'ensemble
nodal est la réunion de $k$ arcs partant de $p$, de courbure géodésique nulle
en $p$ et dont l'extension par réflexion par rapport au bord forme
un système équiangulaire. 
%De plus deux zéros situés 
%sur une même composante connexe du bord de Neumann n'appartiennent pas à la 
%même composante connexe du graphe nodal.
\end{enumerate}
\end{theo}
On appliquera en particulier les propriétés d'incompressibilité au cas du 
disque. On obtient alors :
\begin{cor}\label{nodal:cor}
Si $M$ est homéomorphe à un disque, alors les domaines nodaux sont 
homéomorphes à des disques et les composantes connexes du graphe nodal 
sont des arbres.
\end{cor}
\begin{proof}[du théorème~\ref{nodal:th}]
Comme $\Delta f=0$ dans l'intérieur de $M$, 
on peut appliquer le résultat de S.~Y.~Cheng (théorème~2.5 de~\cite{ch76}).
En particulier,
l'ensemble nodal est la réunion de courbes immergées (les lignes nodales)
qui sont localement en nombre fini, cette réunion étant homéomorphe
à un graphe localement fini.
Cependant, comme l'intérieur de la surface n'est pas compact, on doit 
vérifier la finitude globale du graphe nodal, qui découle des deux points 
suivants:
\begin{itemize}
\item[(i)] le nombre de lignes nodales est fini;
\item[(ii)] les points d'intersection des lignes nodales sont en nombre fini.
\end{itemize}
Le point (i) se déduit du théorème de Courant : le nombre de domaines 
délimités par un ensemble de lignes nodales est au moins égal au nombre
de ces lignes ; par conséquent le nombre total de lignes nodales est 
majoré par le nombre de domaines nodaux, en particulier il est fini.
On montre le point (ii) à l'aide de la formule d'Euler-Poincaré appliquée
à la surface : comme les domaines nodaux sont en nombre fini et que
leur caractéristique d'Euler est majorée par~1, la caractérisque
d'Euler du graphe nodal est minorée en fonction de la topologie de 
la surface et du nombre de domaines. Or, les sommets du graphe sont
de deux types : d'une part les extrémités des lignes nodales, qui 
sont en nombre fini et de degré~1 ; d'autre part les intersections
de lignes, qui sont de degré au moins~4. Si les sommets de degré~4 sont
en nombre infini, la caractéristique d'Euler du graphe serait donc $-\infty$,
ce qui contredit la formule d'Euler-Poincaré.

On en déduit de ce qui précède les points~2 et~4 du théorème et le fait
que le graphe nodal est fini.

Soit $D$ un domaine nodal et $\gamma$ une courbe de $D$ non contractile
dans $D$. Si $\gamma$ est contractile dans $M$, alors elle entoure un 
domaine nodal $D'$ distinct de $D$. En outre, $\gamma$ sépare $D'$ de
$\partial M$, ce qui contredit le lemme~\ref{nodal:lem}. Par conséquent,
$D$ est incompressible. Le même argument montre l'incompressibilité
du graphe nodal. On obtient ainsi les points~1 et~3.

Montrons le point~5. Supposons que $p$ est un point du bord qui n'est
pas l'extrémité d'une ligne nodale. Le point $p$ n'est donc pas situé
à la frontière entre deux domaines nodaux, il est contenu dans un domaine
nodal $D$ sur lequel on supposera que $f$ est positive. D'après le théorème 
d'unique prolongement~\ref{rappels:prol} il n'y a pas de ligne nodale 
le long du bord, on peut donc trouver un petit voisinage $U$ de $p$ délimité 
par une courbe de niveau $f(x)=\varepsilon$ avec $\varepsilon>0$ petit.
En restriction à $D$, $f$ est la première fonction propre du problème
de Steklov-Dirichlet avec condition de Dirichlet sur les ligne nodales
qui bordent $D$ à l'intérieur de $M$. Or, si on définit la fonction test
$\tilde f$ par $\tilde f=\varepsilon$ sur $U$ et $\tilde f=f$ sur
$D\backslash U$, le quotient de Rayleigh de $\tilde f$ est strictement plus
petit que celui de $f$, ce qui contredit que $f$ soit la première fonction
propre sur $D$. 

Le fait qu'un nombre pair de lignes nodales rejoigne chaque composante du bord
découle du fait que le signe de la fonction propre change chaque fois
qu'on traverse une ligne nodale.

Reste à traiter le cas du problème de Steklov-Neumann. 
On considère deux copies de la variété $M$ 
qu'on recolle de manière symétrique le long du
bord de Neumann et on note $M'$ la surface obtenue.
Comme le problème de Steklov est conformément invariant en dimension~2,
on peut lisser la métrique le long du recollement de manière conforme 
et symétrique. Les fonctions propres sur $M$ correspondent alors aux
fonctions propres sur $M'$ qui sont symétriques. On peut en particulier
leur appliquer les résultats de Cheng (point~4). La symétrie de la fonction
implique la symétrie des lignes nodales (sur $M'$) au voisinage du 
bord de Neumann de $M$. On doit encore montrer que le bord de Neumann
ne contient pas de ligne nodale: une fonction harmonique $f$ sur $M'$ est 
localement la partie réelle d'une 
fonction holomorphe $g$ (en munissant $M'$ de la structure complexe 
induite par la structure conforme). Si $f$ est une fonction propre symétrique,
alors la condition de Neumann et l'équation de Cauchy-Riemann implique
que $\mathrm{Im}(g)$ est constante le long du bord de Neumann de $M$. On peut
choisir $g$ de sorte que cette constante soit nulle, les zéros de $f$ sur
le bord de Neumann sont donc les zéros d'une fonction holomorphe. Par 
conséquent ils sont isolés.

%Enfin, si si deux zéros de $f$ situés sur une même composante du
%bord de Neumann appartiennent à une même composante du graphe 
%nodal, la ligne nodale qui les rejoint va délimiter au moins un domaine qui 
%ne recontrera pas le bord de Steklov, ce qui contredit le 
%lemme~\ref{nodal:lem}.
\end{proof}

\begin{proof}[du corollaire~\ref{nodal:cor}]
Si $M$ est un disque, l'incompressibilité des domaines nodaux implique
qu'ils sont simplement connexes, donc que ce sont des disques. 

Les composantes connexes du graphe nodal sont planaires, et leur
incompressiblité signifie qu'ils sont sans cycle. Donc ce sont des arbres.
\end{proof}

\subsection{Bornes sur la multiplicité}
On peut maintenant démontrer les théorèmes~\ref{intro:th2}, 
\ref{intro:th3} et~\ref{intro:th4}. En ce qui concerne le 
théorème~\ref{intro:th2},
on reprendra les arguments de~\cite{ch76} et~\cite{be80}, qui sont
moins précis que ceux de~\cite{na88} mais plus faciles à adapter
au problème de Steklov.

\begin{proof}[du théorème~\ref{intro:th2}]
Supposons que la surface $M$ soit orientable. On note $E_k$ l'espace propre 
associé à la valeur propre $\sigma_k(M)$ et $m_k$ sa
multiplicité. Selon~\cite{be80}, si $m_k>4\gamma+2k+1$, il existe un point 
$x$ dans l'intérieur de $M$ et une fonction propre $f\in E_k$ telle que 
l'ordre d'annulation de $f$ en $x$ soit strictement supérieur à $2\gamma+k$.
Localement, il existe donc au moins $4\gamma+2k+2$ arcs nodaux partant de $p$.

 Si on «~ferme~» la surface en quotientant chaque composante du bord 
sur un point, tous les arcs nodaux se referment, et il existe donc au moins
$2\gamma+k+1$ lacets distincts $C^1$ par morceaux dans l'ensemble nodal. Or,
Cheng a montré (\cite{ch76}, lemme~3.1) que ces lacets décomposent la 
surface en au moins $k+2$ composantes connexes. La fonction $f$ possède
donc au moins $k+2$ domaines nodaux, ce qui contredit le théorème de Courant.

Comme dans \cite{be80}, le cas des surfaces non orientables se traite
par passage à un revêtement à deux feuillets. La surface obtenue
en quotientant les bords est de caractéristique d'Euler $p=(1-\chi(M)-l)$.
Les arguments des \cite{be80} donnent alors la majoration $m_k\leq 4p+4k+3$.
\end{proof}

\begin{proof}[du théorème~\ref{intro:th3}]
On note $E$ l'espace propre associé à la valeur propre $\sigma_i(M,\rho,g)$
pour $i=1$ ou $2$.

Soit $p$ un point intérieur au disque. Si $E$ est de
dimension au moins~4, il existe une fonction propre non nulle $f\in E$
telle que $f$ et $\de f$ soient nuls en $p$. Le point $p$ est 
donc un sommet du graphe nodal de $f$ et il en part au moins quatre arêtes.
Comme le graphe nodal est un arbre dont les feuilles sont sur le bord, il
délimite au moins quatre domaines nodaux. Il y a donc contradiction avec 
le théorème de Courant.

Supposons maintenant que $i=1$ et que $E$ soit de dimension~3.
Si $p_0$ est un point du bord, le sous-espace des fonctions $f\in E$ telles
que $f(p_0)=0$ est de dimension au moins~2. Pour tout point $p$ du bord
distinct de $p_0$, il existe donc une fonction $f_p$, qu'on supposera
de norme~1, telle que $f(p)=f(p_0)=0$. Comme chacun de ces points
est nécessairement l'extrémité d'un ligne nodale et que la fonction
$f_p$ a exactement deux domaines nodaux, $p_0$ et $p$ sont les 
extrémités de l'unique ligne nodale de $f_p$.

Le bord est donc partagé
en deux intervalles, $I_p^+$ et $I_p^-$, d'extrémités $p_0$ et $p$, sur 
lesquels la fonction $f_p$ est respectivement positive et négative. En 
faisant tendre $p$ vers $p_0$,
on peut faire tendre la longueur de $I_p^-$ vers~0. Comme les fonctions
$f_p$ sont normées et que $E$ est de dimension finie, la famille $f_p$ 
admet une limite $f$ (quitte à extraire une sous-famille). La fonction
$f$ est alors positive ou nulle sur la totalité du bord, puisque l'évaluation
en un point est une forme linéaire continue sur $E$. Par conséquent, $f$
est de signe constant, ce qui est impossible puisque c'est une fonction
propre de la valeur propre $\sigma_1$.
\end{proof}

\begin{proof}[du théorème~\ref{intro:th4}]
On note $I$ une composante connexe du bord de Neumann $\partial\mathbb D_N$
et on choisit $k+1$ points distincts $x_1,\ldots,x_{k+1}$ dans $I$.
Supposons que la multiplicité de $\sigma_k(\mathbb D, \partial\mathbb D_S,
\rho,g)$ soit supérieure ou égale à $k+2$. On peut alors trouver
une fonction propre $f$ associée à cette valeur propre qui s'annule en 
tous les points $x_i$, $i\in[1,k+1]$. En vertu du point~6 du 
théorème~\ref{nodal:th}, chaque $x_i$ appartient à une composante connexe
du graphe nodal qui joint $x_i$ à un point du bord de Steklov, ces 
composantes étant distinctes. L'ensemble nodal sépare donc $\mathbb D$ en
au moins $k+2$ composantes connexes, ce qui contredit  le théorème 
de Courant.
\end{proof}

\noindent Pierre \textsc{Jammes}\\
Laboratoire J.-A. Dieudonné\\
Université Nice Sophia Antipolis --- CNRS (UMR 7351)\\
Parc Valrose \\
06108 Nice Cedex 02, France\\
\texttt{pjammes@unice.fr}

\end{document}